\documentclass[11pt,a4paper]{article}

\usepackage{epsf,epsfig,amsfonts,amsgen,amsmath,amstext,amsbsy,amsopn,amsthm
}
\usepackage{amsmath}
\usepackage{amsfonts,amsthm,amssymb}
\usepackage{amsfonts}
\usepackage{graphics}
\usepackage{latexsym,bm}
\usepackage{amsfonts,amsthm,amssymb,bbding}
\usepackage{indentfirst}
\usepackage{graphicx}
\usepackage{color}
\newtheorem{theorem}{Theorem}[section]

\newtheorem{prop}[theorem]{Proposition}
\newtheorem{lemma}[theorem]{Lemma}
\newtheorem{false statement}{False statement}

\newtheorem{definition}[theorem]{Definition}
\newtheorem{claim}[theorem]{Claim}

\newtheorem{remark}[theorem]{Remark}
\newtheorem{conj}[theorem]{Conjecture}

\newtheorem{prob}[theorem]{Problem}

\newcommand{\la}{\lambda}
\begin{document}
\begin{sloppypar}
\title{\bf\Large  Eigenvalues and triangles in graphs}
\author{\Large Huiqiu Lin$^{a}$\thanks{Department of
Mathematics, East China University of Science and
Technology, Shanhai 200237, P.R. China. Email:
huiqiulin@126.com. Research supported by NSFC
(Nos. \ 11771141 and \ 12011530064).}~~~Bo
Ning$^{b}$\thanks{School of Mathematics,
Tianjin University, Tianjin 300072, P.R. China.}
\thanks{Corresponding author. The current
address: College of Computer Science,
Nankai University, Tianjin 300071,~P.R.~China. E-mail:
bo.ning@nankai.edu.cn. Research supported by NSFC
(Nos. 11601379 and 11971346).}~~~and Baoyindureng
Wu$^{c}$\thanks{College of Mathematics and System
Science, Xinjiang University, Urumqi, Xinjiang
830046, P.R. China. E-mail: wubaoyin@hotmail.com.
Research supported by NSFC (No. 11571294).}}
\date{}

\maketitle
\begin{center}
\begin{minipage}{140mm}
\small\noindent{\bf Abstract:}
Bollob\'as and Nikiforov [J. Combin. Theory, Ser. B. 97
(2007) 859--865] conjectured the following. If $G$ is
a $K_{r+1}$-free graph on at least $r+1$ vertices and $m$
edges, then $\la^2_1(G)+\la^2_2(G)\leq \frac{r-1}{r}\cdot2m$,
where $\la_1(G)$ and $\la_2(G)$ are the largest and the
second largest eigenvalues of the adjacency matrix $A(G)$,
respectively. In this paper, we confirm the conjecture
in the case $r=2$, by using tools from doubly stochastic
matrix theory, and also characterize all families of
extremal graphs. Motivated by classic theorems due to
Erd\H{o}s and Nosal respectively, we prove that every
non-bipartite graph $G$ of order $n$ and size $m$
contains a triangle, if one of the following is true:
(1) $\la_1(G)\geq\sqrt{m-1}$ and $G\neq C_5\cup (n-5)K_1$; and (2)
$\la_1(G)\geq \la_1(S(K_{\lfloor\frac{n-1}{2}\rfloor,\lceil\frac{n-1}{2}\rceil}))$
and $G\neq S(K_{\lfloor\frac{n-1}{2}\rfloor,\lceil\frac{n-1}{2}\rceil})$,
where $S(K_{\lfloor\frac{n-1}{2}\rfloor,\lceil\frac{n-1}{2}\rceil})$
is obtained from $K_{\lfloor\frac{n-1}{2}\rfloor,\lceil\frac{n-1}{2}\rceil}$
by subdividing an edge. Both conditions are best possible.
We conclude this paper with some open problems.


\end{minipage}
\end{center}

\renewcommand{\thefootnote}{}
\vskip 0.3in

\section{Introduction}
It is well known that spectra of graphs can be used
to describe structural properties and parameters of
graphs, including cycle structures \cite{H}, maximum
cuts \cite{A96}, matchings and factors \cite{BH05},
regularity \cite{LSW}, diameter \cite{C} and expander
properties \cite{A86}, etc. Recently, there has been
extensive research in the literature
(see \cite{CGH,CDKL,L,LN,TT,FZL,LLD}). Referring to
spectral extremal graph theory, a bulk of related
work was included in the detailed survey \cite{N11},
see references therein.

In this paper, we focus on spectral extremal graph
theory and mainly investigate the relationship
between triangles and eigenvalues of the adjacency
matrix of a graph. Throughout this paper, let $G$
be a graph with order $v(G):=n$, size $e(G):=m$
and clique number $\omega(G):=\omega$. Let $A(G)$
be its adjacency matrix. The eigenvalues
$\la_1(G):=\la_1\ge\la_2(G):=\la_2\ge\cdots\ge\la_n(G)$
of $A(G)$ are called the \emph{eigenvalues} of $G$.
For all integers $n\geq 1$, we set $[n]=\{1,2,\ldots,n\}$.

The study of bounding spectral radius of a graph in
terms of some parameters has a rich history. Starting
from 1985, Brualdi and Hoffman \cite{BH} proved that
$\lambda_1\leq k-1$ if $m\leq \binom{k}{2}$ for some
integer $k\geq 1$. This result was extended by
Stanley \cite{S} who showed that
$\la_1\leq\frac{1}{2}(\sqrt{8m+1}-1)$. The bound
is best possible for complete graphs (possibly with
isolated vertices), but still can be improved for
special classes of graphs, such as triangle-free
graphs (see Nosal \cite{N70}). For further
generalizations and related extensions of Stanley's
result, see Hong \cite{H93}, Hong, Shu and Fang \cite{HSF},
Nikiforov \cite{N02}, and Zhou and Cho \cite{ZC}.
Specific to bounding spectral radius of a graph in
terms of clique number, Wilf \cite{W} showed that
$\la_1\leq \frac{\omega-1}{\omega}n$. A better
inequality $\la_1\leq \sqrt{\frac{2(\omega-1)m}{\omega}}$,
implicitly conjectured by Edwards and Elphik \cite{EE},
was confirmed by Nikiforov in \cite{N02} using a
technique of Motzkin and Straus \cite{MS}. Later,
the extremal graphs when equality holds were
characterized in \cite{N06}. By the inequality
$\la_1\geq \frac{2m}{n}$, one can easily deduce
the concise form of Tur\'an's theorem that
$m\leq \frac{\omega-1}{2\omega}n^2$ from
Nikiforov's inequality. Therefore, Nikiforov's
inequality sometimes is called spectral Tur\'an's
theorem.

In 2007, Bollob\'{a}s and Nikiforov \cite{BN} posed
the following nice conjecture, which is the original
motivation of our article.
\begin{conj} (\cite[Conjecture~1]{BN})\label{Conj-BN}
Let $G$ be a $K_{r+1}$-free graph of order at least $r+1$
with $m$ edges. Then
$$\lambda_1^2+\lambda_2^2\leq \frac{r-1}{r}\cdot2m.$$
\end{conj}
Notice that Conjecture \ref{Conj-BN}, if true, will
improve Nikiforov's inequality. To our knowledge, the
conjecture is still open now. In this paper, we make
the first progress on this conjecture. In fact, we
solve the case $r=2$ by using tools from doubly
stochastic matrix theory, and also characterize
all extremal graphs.

Let $G$ be a graph. A ``\emph{blow-up}" of $G$ is a
new graph obtained from $G$ by replacing each vertex
$x\in V(G)$ by an independent set $I_x$, in which for
any two vertices $x,y\in V(G)$, we add all edges between
$I_x$ and $I_y$ if $xy\in E(G)$. Let $P_n$ denote a path
on $n$ vertices, that is, a path of length $n-1$. For an
integer $k\geq 2$, $kP_n$ denotes the disjoint union of
$k$ copies of $P_n$.
\begin{theorem}\label{Thm-BN}
Let $G$ be a triangle-free graph of order at least $3$
with $m$ edges. Then $$\lambda_1^2+\lambda_2^2\leq m,$$
where equality holds if and only if $G$ is a blow-up
of some member of $\mathcal{G}$, in which
$\mathcal{G}=\{P_2\cup K_1,2P_2\cup K_1,P_4\cup K_1,P_5\cup K_1\}$.
\end{theorem}

Recall that a quintessential result in extremal graph
theory is Mantel's theorem, which maximizes the number
of edges over all triangle-free graphs. We would emphasize
that Theorem \ref{Thm-BN} strengthens the spectral
strengthening of Mantel's theorem due to Nosal \cite{N70},
which states that every triangle-free graph $G$ on $m$ edges
satisfies that $\la_1\leq \sqrt{m}$. On the other hand, Mantel's
theorem was improved by Erd\H{o}s (see \cite[Ex.12.2.7]{BM}) in the
following form: every non-bipartite triangle-free graph of
order $n$ and size $m$ satisfies that $m\leq\frac{(n-1)^2}{4}+1$.
Notice that a subdivision of
$K_{\lfloor\frac{n-1}{2}\rfloor,\lceil\frac{n-1}{2}\rceil}$
on one edge shows the upper bound is tight. In this article,
we shall prove spectral versions of Erd\H{o}s' theorem.

Our two results are the following.

\begin{theorem}\label{Thm-trianglesize}
Let $G$ be a non-bipartite graph with size $m$. If
$\la_1\geq \sqrt{m-1}$, then $G$ contains a triangle,
unless $G$ is a $C_5$ (possibly together with some
isolated vertices).
\end{theorem}

\begin{theorem}\label{Thm-triangleorder}
Let $G$ be a non-bipartite graph with order $n$. If
$\lambda_1\geq \lambda_1(S(K_{\lfloor\frac{n-1}{2}\rfloor,\lceil\frac{n-1}{2}\rceil}))$
where $S(K_{\lfloor\frac{n-1}{2}\rfloor,\lceil\frac{n-1}{2}\rceil})$
denotes a subdivision of $K_{\lfloor \frac{n-1}{2}\rfloor,\lceil \frac{n-1}{2}\rceil}$
on one edge, then $G$ contains a triangle unless
$G\cong S(K_{\lfloor\frac{n-1}{2}\rfloor,\lceil\frac{n-1}{2}\rceil})$.
\end{theorem}

\section{The Bollob\'{a}s-Nikiforov Conjecture for triangle-free graphs}\label{Sec:2}
In Section \ref{Sec:2}, we introduce necessary
preliminaries for doubly stochastic matrix theory
and then prove Theorem \ref{Thm-BN}. For more
details on related knowledge, we refer the reader
to Zhan \cite{Zhan}.

A nonnegative square matrix is called \emph{doubly
stochastic}, if every entry is at least 0 and the
sum of the entries in every row and every column
is 1, and called \emph{doubly substochastic} if
the sum of the entries in every row and every
column is less than or equal to 1. A square matrix
is called a \emph{weak-permutation matrix} if every
row and every column has at most one nonzero entry
and all the nonzero entries (if existing) are $1$.

We also use the definition of ``\emph{a vector is
weakly majorized by the other one}" as follows,
where we rearrange the components of
$x=(x_1,x_2,\ldots, x_n),y=(y_1,y_2,\ldots,y_n)\in \mathbb{R}^n$
in non-increasing order as
$x_{[1]}\ge x_{[2]}\ge \cdots\ge x_{[n]}$
and $y_{[1]}\ge y_{[2]}\ge \cdots\ge y_{[n]}$.

\begin{definition}
Let $x=(x_1,x_2,\ldots, x_n),y=(y_1,y_2,\ldots, y_n)\in \mathbb{R}^n$. If
$$\sum_{i=1}^kx_{[i]}\le \sum_{i=1}^ky_{[i]},~~k=1,2,\ldots,n,$$
then we say that $x$ is \emph{weakly majorized by}
$y$ and denote it by $x\prec_w y$. If $x\prec_w y$
and $\sum_{i=1}^nx_i=\sum_{i=1}^ny_i$, then we say
that $x$ is \emph{majorized by} $y$ and denote it by
$x \prec y$.
\end{definition}
The following lemma is a basic property on a doubly
substochastic matrix.
\begin{lemma}\label{Le-1}
{\rm(\cite[Lemma~3.24]{Zhan})}
Let $x,y\in \mathbb{R}^n_+=\{(z_1,\ldots,z_n)|z_i\geq 0,1\leq i\leq n\}$.
Then $x\prec_w y$ if and only if there exists a
doubly substochastic matrix $A$ such that $x=Ay$.
\end{lemma}

One of the main ingredients in our proof is using
the relationship between doubly (sub)stochastic
matrix and (weak-)permutation matrix.
\begin{lemma}\label{Le-2}
{\rm(\cite[Theorem~3.22]{Zhan})}
Every doubly substochastic matrix is a convex
combination of weak-permutation matrices.
\end{lemma}

The following theorem will play an essential
role in our proof of Theorem \ref{Thm-BN},
whose proof uses Minkowski's inequality
(see \cite[pp.8~]{K}).

\begin{lemma}[Minkowski's inequality]\label{Lem-Mink}
Let $x,y\in \mathbb{R}^n_+$. If $p>1$, then
$\|x+y\|_p\leq \|x\|_p+\|y\|_p$. Moreover, if
$x,y\in \mathbb{R}^n_+$, $x\neq\theta$, and
$y\neq\theta$ where $\theta=(0,0,\ldots,0)$,
then equality holds if and only if there exists
$\alpha>0$ such that $x=\alpha y$.
\end{lemma}
By induction on $k$ (see below), Lemma \ref{Lem-Mink}
can be extended to a multiple version easily.

\begin{lemma}[Multiple Minkowski's inequality]\label{Lem-MultMink}
Let $k\in \mathbb{Z}$, $k\geq 2$, and $x^i\in \mathbb{R}^n_+$
where $i\in [k]$. If $p>1$ then
$\|\sum\limits_{i=1}^k x^i\|_p\leq \sum\limits_{i=1}^k\|x^i\|_p$.
Moreover, if $x^i\neq \theta$ for all $i$, then equality
holds if and only if there exists $\alpha_{i,j}>0$ such
that $x^i=\alpha_{i,j} x^j$ for all $i,j\in [n]$ with
$i\neq j$.
\end{lemma}

\begin{theorem}\label{na}
Let $x=(x_1,x_2,\ldots,x_n),y=(y_1,y_2,\ldots,y_n)\in \mathbb{R}^n_+$
such that $\{x_i\}_{i=1}^n$ and $\{y_i\}_{i=1}^n$
are in non-increasing order. If $y\prec_w x$, then
$\|y\|_p\leq\|x\|_p$ for every real number $p>1$,
where equality holds if and only if $x=y$.
\end{theorem}
\begin{proof}
If $x=\theta$, then $y=\theta$. Now assume $x\neq \theta$.
Since $y\prec_w x$, there exists a doubly substochastic
matrix $A$ such that $y=Ax$ by Lemma \ref{Le-1}. By
Lemma \ref{Le-2}, there are $s$ weak-permutation matrices
$P_i$ for all $i\in [s]$, such that $A=\sum_{i=1}^sa_iP_i$,
where $\sum_{i=1}^sa_i=1$, $a_i\geq 0$. Without loss
of generality, we can assume $a_i>0$ for all $i\in [s]$.
Notice that $y=Ax=(\sum_{i=1}^sa_iP_i)x=\sum_{i=1}^sa_i(P_ix)$.
Therefore,
\begin{align*}
\|y\|_p=\|\sum_{i=1}^sa_i(P_ix)\|_p
\le \sum_{i=1}^sa_i\|P_ix\|_p\le\sum_{i=1}^sa_i\|x\|_p
=(\sum_{i=1}^sa_i)\cdot\|x\|_p=\|x\|_p.
\end{align*}

If $x=y$, then obviously $\|x\|_p=\|y\|_p$. If
$\|x\|_p=\|y\|_p$, then
\begin{align}\label{equ}
\|\sum_{i=1}^sa_i(P_ix)\|_p= \sum_{i=1}^sa_i\|P_ix\|_p=\sum_{i=1}^sa_i\|x\|_p.
\end{align}
Since $\|P_ix\|_p\leq\|x\|_p$, from (\ref{equ}), we obtain $\|P_ix\|_p=\|x\|_p>0$
for all $i\in [s]$, and so $a_iP_ix\neq \theta$ for each $i$.
By Lemma \ref{Lem-MultMink}, the first equality of (\ref{equ})
implies that for any pair of distinct integers $i,j\in [s]$,
there exists a real number $\alpha_{i,j}>0$ such that
$a_i(P_ix)=\alpha_{i,j}a_j(P_jx)$. By the second equality
of (\ref{equ}), since each $a_i>0$, we have
$||P_ix||_p=\|x\|_p=||P_jx||_p$. Then $\alpha_{i,j}\cdot a_j=a_i\neq 0$.
Thus, $P_ix=P_jx$ and moreover $P_ix=P_1x$ for each $i\in [s]$.
It follows that $y=\sum_{i=1}^sa_i(P_ix)=P_1x$.

Since $y=P_1x$ and $\|x\|_p=\|y\|_p$ where $P_1$ is a
weak-permutation matrix, we know that $\{y_i\}_{i=1}^n$
is just a rearrange of elements of $\{x_i\}_{i=1}^n$.
As both $\{x_i\}_{i=1}^n$ and $\{y_i\}_{i=1}^n$ are
non-increasing sequences, we have $y=x$. The proof
is complete.
\end{proof}
For a graph $G$, the \emph{rank} of $G$, denoted by
$rank(G)$, is defined as the rank of $A(G)$. We need
Theorem 3.3 and Theorem 4.3 in \cite{O} to characterize
the extremal graphs in Theorem \ref{Thm-BN}, and so
list them as a lemma below.
\begin{lemma}{\rm(\cite{O})}\label{lem5}
Let $G$ be a graph with order $n$. Then we have
the following statements.

\noindent (I) If $rank(G)=2$, then $G$ is a blow-up
of $P_2\cup K_1$.

\noindent (II) If $G$ is a bipartite graph with
$rank(G)=4$, then $G$ is a blow-up of $\Gamma$,
where $\Gamma\in\{2P_2\cup K_1,P_4\cup K_1,P_5\cup K_1\}$.
\end{lemma}

We shall give a proof of Theorem \ref{Thm-BN}. Before
it, we define ``\emph{the inertia of a graph}" as the
ordered triple $(n^+,n^-,n^0)$, where $n^+$, $n^-$
and $n^0$ are the numbers (counting multiplicities)
of positive, negative and zero eigenvalues of the
adjacency matrix $A(G)$, respectively.

\noindent
{\bf Proof of Theorem \ref{Thm-BN}.}
Let $n$ be the order of $G$ and $(n^+,n^-,n^0)$ be the inertia
of $G$. Set $s^+:=\la_{1}^2+\cdots+\la^2_{n^+}$ and
$s^-:=\la_{n-n^-+1}^2+\cdots+\la_n^2$. Since $G$ is triangle-free,
we have $G\ncong K_n$, and so $\la_2(G)\geq 0$ (see Lemma 5
in \cite{H88}).

Suppose that $\la_1^2+\la_2^2>m$. Since $s^++s^-=2m$,
we have $\la_1^2+\la_2^2>\frac{s^++s^-}{2}$, and so
$\la_1^2+\la_2^2\geq 2(\la_1^2+\la_2^2)-s^+>s^-\geq0$.
Now, we construct two $n^-$-vectors $x$ and $y$ such
that $x=(\la_1^2,\la_2^2,0, \ldots, 0)^T$ and
$y=(\la_n^2, \la_{n-1}^2,\ldots,\la_{n-n^-+1}^2)^T$.
Since $\la_1^2+\la_2^2>s^{-}$, we have $y\prec_w x$
and $x\neq y$. Set $p=\frac{3}{2}$. By Theorem \ref{na},
we have
$\|x\|_\frac{3}{2}^\frac{3}{2}>\|y\|_\frac{3}{2}^\frac{3}{2}$,
that is,
$\la_1^3+\la_2^3>|\la_n|^3+|\la_{n-1}|^3+\cdots+|\la_{n-n^-+1}|^3$.
It implies that
\begin{eqnarray*}
t(G)=\frac{\la_1^3+\la_2^3+\cdots+\la_{n^+}^3+\la_{n-n^-+1}^3+\cdots+\la_n^3}{6}
\ge\frac{\la_1^3+\la_2^3+\la_{n-n^-+1}^3+\cdots+\la_n^3}{6}>0.
\end{eqnarray*}
This gives us a contradiction. Thus, we proved
$\la_1^2+\la_2^2\leq m$.

If $\la_1^2+\la_2^2=m$, then $\la_1^2+\la_2^2\geq s^-\geq0$.
It follows that $y\prec_w x$. By Theorem \ref{na}, we have
$\|x\|_\frac{3}{2}^\frac{3}{2}\geq\|y\|_\frac{3}{2}^\frac{3}{2}$.
Since $G$ is triangle-free, this implies that
\begin{eqnarray*}
0=t(G)=\frac{\la_1^3+\la_2^3+\cdots+\la_{n^+}^3+\la_{n-n^-+1}^3+\cdots+\la_n^3}{6}
\ge\frac{\la_1^3+\la_2^3+\la_{n-n^-+1}^3+\cdots+\la_n^3}{6}\geq0.
\end{eqnarray*}
Therefore, $\la_1^3+\la_2^3=-(\la_{n-n^-+1}^3+\cdots+\la_n^3)$, which
implies $\|x\|_\frac{3}{2}^\frac{3}{2}=\|y\|_\frac{3}{2}^\frac{3}{2}$.
Again by Theorem \ref{na}, $x=y$. It follows that $\la_1^2=\la_n^2$
and $\la_2^2=\la_{n-1}^2$. Thus $\la_1=-\la_n$ and $\la_2=-\la_{n-1}$.
By the trace formula $\sum_{i=1}^n\la_i=0$, we infer that all the
remaining eigenvalues are 0. If $\la_2=0$, then $rank(G)=2$. By
Lemma \ref{lem5} (I), $G$ is a blow-up of $P_2\cup K_1$. Recall
$\la_1=-\la_n$, which implies that $G$ is bipartite. If $\la_2\neq 0$,
then $rank(G)=4$. By Lemma \ref{lem5} (II) and the fact that $G$
is bipartite, $G$ is a blow-up of $\Gamma$, where $\Gamma$ is
$2P_2\cup K_1$ or $P_4\cup K_1$ or $P_5\cup K_1$. The proof
is complete. {\hfill$\Box$}

By using the same method as in the proof of Theorem \ref{Thm-BN},
we can deduce the following.
\begin{theorem} {\rm(\cite[Theorem~2(i)]{N09})}
Let $G$ be a graph of size $m$. If $\la_1^2\geq m$, then
$G$ contains a triangle, unless $G$ is a blow up of
$P_2\cup K_1$.
\end{theorem}

\section{Proofs of Theorems \ref{Thm-trianglesize}
and \ref{Thm-triangleorder}}
A walk $v_1v_2\cdots v_k$ $(k\geq2)$ in a graph $G$ is
called an \emph{internal path}, if these $k$ vertices
are distinct (except possibly $v_1=v_k$), $d_G(v_1)\geq3$,
$d_G(v_k)\geq3$ and $d_G(v_2) =\cdots= d_G(v_{k-1}) = 2$
(unless $k = 2$). We denote by $G_{uv}$ the graph obtained
from $G$ by subdividing the edge $uv$, that is, introducing
a new vertex on the edge $uv$. Let $Y_n$ be the graph
obtained from an induced path $v_1v_2\cdots v_{n-4}$ by
attaching two pendant vertices to $v_1$ and other two
pendant vertices to $v_{n-4}$.

Hoffman and Smith \cite{HS} proved the following result
(see also Ex.14 in \cite[pp.79]{CDS}), which is used
towards the structure of extremal graphs in
Theorem \ref{Thm-trianglesize}.
\begin{lemma}{\rm(\cite{HS})}\label{Le-HS}
Let $G$ be a connected graph with $uv \in E(G)$. If
$uv$ belongs to an internal path of $G$ and
$G\ncong Y_n$, then $\la_1(G_{uv})<\la_1(G)$.
\end{lemma}

Now we shall prove Theorems \ref{Thm-trianglesize}
and \ref{Thm-triangleorder}.

\begin{figure}[htbp]
\centering
\unitlength 3mm % = 2.845pt
\linethickness{0.4pt}
\ifx\plotpoint\undefined\newsavebox{\plotpoint}\fi % GNUPLOT compatibility
\begin{picture}(28,9)(0,0)
\thicklines
%\emline(3,8)(1,6)
\multiput(3,8)(-.03333333,-.03333333){60}{\line(0,-1){.03333333}}
%\end
\put(1,6){\line(0,-1){3}}
\put(1,3){\line(1,0){4}}
\put(5,3){\line(0,1){3}}
%\emline(5,6)(3,8)
\multiput(5,6)(-.03333333,.03333333){60}{\line(0,1){.03333333}}
%\end
\put(3,8){\line(0,1){0}}
\put(3,8){\line(1,0){2}}
%\emline(11,8)(9,6)
\multiput(11,8)(-.03333333,-.03333333){60}{\line(0,-1){.03333333}}
%\end
\put(9,6){\line(0,-1){3}}
\put(9,3){\line(1,0){4}}
\put(13,3){\line(0,1){3}}
%\emline(13,6)(11,8)
\multiput(13,6)(-.03333333,.03333333){60}{\line(0,1){.03333333}}
%\end
\put(11,8){\line(0,1){0}}
\put(11,8){\line(1,0){5}}
\put(16,8){\line(-3,-5){3}}
\put(13,3){\line(0,1){0}}
\put(13,3){\line(0,1){0}}
\put(13,6){\line(1,-1){4}}
%\emline(17,2)(9,3)
\multiput(17,2)(-.2666667,.0333333){30}{\line(-1,0){.2666667}}
%\end
\put(9,3){\line(0,1){0}}
%\emline(23,8)(21,6)
\multiput(23,8)(-.03333333,-.03333333){60}{\line(0,-1){.03333333}}
%\end
\put(21,6){\line(0,-1){3}}
\put(21,3){\line(1,0){4}}
\put(25,3){\line(0,1){3}}
%\emline(25,6)(23,8)
\multiput(25,6)(-.03333333,.03333333){60}{\line(0,1){.03333333}}
%\end
\put(23,8){\line(0,1){0}}
\put(23,8){\line(1,0){4}}
\put(27,8){\line(-2,-5){2}}
\put(25,3){\line(0,1){0}}
\put(23,8){\line(-1,0){4}}
\put(19,8){\line(2,-5){2}}
\put(21,3){\line(0,1){0}}

%\end

\put(5,8){\circle*{.5}}
\put(3,8){\circle*{.5}}
\put(1,6){\circle*{.5}}
\put(1,3){\circle*{.5}}
\put(5,3){\circle*{.5}}
\put(5,6){\circle*{.5}}
\put(11,8){\circle*{.5}}
\put(9,6){\circle*{.5}}
\put(9,3){\circle*{.5}}
\put(13,3){\circle*{.5}}
\put(13,6){\circle*{.5}}
\put(16,8){\circle*{.5}}
\put(17,2){\circle*{.5}}
\put(23,8){\circle*{.5}}
\put(21,6){\circle*{.5}}
\put(21,3){\circle*{.5}}
\put(25,3){\circle*{.5}}
\put(25,6){\circle*{.5}}
\put(27,8){\circle*{.5}}
\put(19,8){\circle*{.5}}

\put(3,8.5){\makebox(0,0)[cc]{\tiny$u_1$}}
\put(0.3,6){\makebox(0,0)[cc]{\tiny$u_2$}}
\put(0.3,3){\makebox(0,0)[cc]{\tiny$u_3$}}
\put(5.9,3){\makebox(0,0)[cc]{\tiny$u_4$}}
\put(5.9,6){\makebox(0,0)[cc]{\tiny$u_5$}}
\put(5.9,8){\makebox(0,0)[cc]{\tiny$v$}}
\put(11,8.5){\makebox(0,0)[cc]{\tiny$u_1$}}
\put(8.3,6){\makebox(0,0)[cc]{\tiny$u_2$}}
\put(8.3,3){\makebox(0,0)[cc]{\tiny$u_3$}}
\put(13.9,3){\makebox(0,0)[cc]{\tiny$u_4$}}
\put(13.9,6){\makebox(0,0)[cc]{\tiny$u_5$}}
\put(16,8.5){\makebox(0,0)[cc]{\tiny$v$}}
\put(18,1.7){\makebox(0,0)[cc]{\tiny$w$}}
\put(23,8.5){\makebox(0,0)[cc]{\tiny$u_1$}}
\put(20.3,6){\makebox(0,0)[cc]{\tiny$u_2$}}
\put(20.3,3){\makebox(0,0)[cc]{\tiny$u_3$}}
\put(25.7,3){\makebox(0,0)[cc]{\tiny$u_4$}}
\put(25.7,6){\makebox(0,0)[cc]{\tiny$u_5$}}
\put(28,8){\makebox(0,0)[cc]{\tiny$v$}}
\put(18,8){\makebox(0,0)[cc]{\tiny$w$}}
\put(3,0){\makebox(0,0)[cc]{\small$H_1$}}
\put(12,0){\makebox(0,0)[cc]{\small$H_2$}}
\put(23,0){\makebox(0,0)[cc]{\small$H_3$}}
\end{picture}
  \caption{The graphs $H_1$, $H_2$ and $H_3$.}\label{fig-1}
\end{figure}
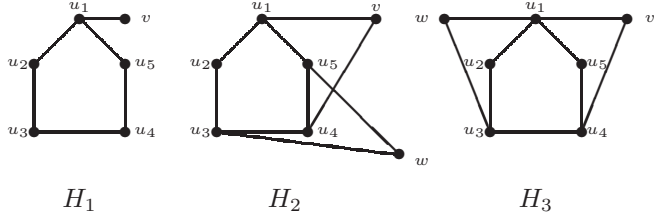

\noindent
{\bf Proof of Theorem \ref{Thm-trianglesize}.}
Suppose to the contrary that $G$ contains no triangles.
Assume $\la_2\geq 1$. Then $\la_1^2+\la_2^2\geq m$.
Since $G$ is non-bipartite, by Theorem \ref{Thm-BN},
$G$ contains a triangle, a contradiction. Now assume
$\la_2< 1$. This implies that, if $G$ is disconnected
then its every component is an isolated vertex except
for one component.

We consider the case that $G$ is connected.
Let $s$ be the length of a shortest odd
cycle of $G$, where $s\geq 5$. Note that
$\la_2(C_{s})=2\cos\frac{2\pi}{s}$. If $s\geq 6$,
then $\la_2(C_{s})\geq 1$, and by Cauchy's
interlacing theorem, $\la_2(G)\geq\la_2(C_s)=1$,
a contradiction. Thus $s=5$. Let
$S=\{u_i:1\leq i\leq 5\}\subseteq V(G)$
with $G[S]=u_1u_2u_3u_4u_5u_1$. If $n=5$,
then $G\cong C_5$, and we are done.
Let $T=N(S)\backslash S$.

We shall use the property that $G$ contains
no $H_i$ as an induced subgraph where $i=1,2,3$,
since $\lambda_2(H_i)=1>\la_2(G)$ (recall $H_i$'s
in Figure 1). In the following, we say that $G$
is $H$-free if it contains no $H$ as an induced
subgraph.

We first claim that $d_S(v)=2$ for each $v\in T$.
For $v\in T$, without loss of generality, assume
that $v\in N(u_1)$. If $d_S(v)\geq 3$, then there
exists $i\in [5]$ such that $vu_i,vu_{i+1}\in E(G)$,
where the subscripts $i,i+1$ are taken modulo 5 and
$u_0=u_5$. In this case, there is a triangle
$vu_iu_{i+1}v$ in $G$, a contradiction.
If $d_S(v)=1$, then $N_S(v)=\{u_1\}$ and
$\{v,u_1,u_2,u_3,u_4,u_5\}$ induces an $H_1$,
a contradiction. This shows that $d_S(v)=2$ for each
$v\in T$. Next, we claim that $V(G)=S\cup T$. Indeed,
if not, there exists at least one vertex, say $v'$,
which is at distance 2 from $S$.
We can assume that $v'vu_1$ is an induced $P_3$ such
that $v'u_i\notin E(G)$ for any $i\in [5]$. Since
$d_S(v)=2$, by symmetry, we can assume
$N_S(v)=\{u_1,u_3\}$. Since $G$ is triangle-free
and $v'u_i\notin E(G)$ for all $i\in [5]$,
we can find that $\{v',v,u_3,u_4,u_5,u_1\}$
induces an $H_1$, a contradiction. This shows
that $V(G)=S\cup T$.

We choose $v\in T$ and assume $N_{S}(v)=\{u_1,u_3\}$
(by symmetry). Recall that $n\geq 6$. If $n=6$,
then $m=7$, and by a simple calculation,
$\la_1(G)=2.3914<\sqrt{6}$, a contradiction.
Therefore, $n\geq 7$ and this implies
$V(G)\backslash (S\cup \{v\})\neq \emptyset$.
It follows that $T\backslash \{v\}\neq \emptyset$.

Let $w\in T\backslash \{v\}$. If $N_{S}(w)=N_{S}(v)$
then $wv\notin E(G)$; if $N_{S}(w)\neq N_{S}(v)$,
then $N_{S}(w)\cap N_{S}(v)=\emptyset$, since $G$
is $H_3$-free and triangle-free. Thus,
$N_S(w)=\{u_1,u_3\}$, or $N_{S}(w)=\{u_2,u_4\}$,
or $N_{S}(w)=\{u_2,u_5\}$. Furthermore, for each
of the latter two cases, we have $vw\in E(G)$ since
$G$ is $H_2$-free. Indeed, if $N_{S}(w)=\{u_2,u_4\}$,
since $G$ is triangle-free and $H_3$-free, every
vertex in $T$ is adjacent to $u_1$ and $u_3$, or
to $u_2$ and $u_4$; if $N_S(w)=\{u_2,u_5\}$, then
every vertex in $T$ is adjacent to $u_1$ and $u_3$,
or to $u_2$ and $u_5$.

In the following, we assume $N_{S}(w)=\{u_2,u_4\}$.
Let $A=N_G(u_1)\cap N_G(u_3)$ and
$B=N_G(u_2)\cap N_G(u_4)$. By reasoning the analysis
above, we infer that both $A$ and $B$ are independent
sets, $A\cup B=T\cup \{u_2,u_3\}$, and $G[A\cup B]$
is a complete bipartite subgraph. Let $|A|=a$ and
$|B|=b$. Then $m=ab+(a+1)+(b+1)=(a+1)(b+1)+1$, and
$G$ is a subdivision of $K_{a+1,b+1}$ on some edge.
By Lemma \ref{Le-HS},
$\la_1(G)<\la_1(K_{a+1,b+1})=\sqrt{(a+1)(b+1)}=\sqrt{m-1}$,
a contradiction.

If $G$ is disconnected, then there is only one
non-trivial component. We apply the conclusion obtained above
to the component, and shall get $G=C_5\cup (n-5)K_1$,
where $n$ is the order of $G$. The proof is complete. {\hfill$\Box$}

In the following, we use $S(G)$ to denote a subdivision
of $G$ on an edge, if the subdivision is unique up to
isomorphic. Proof of Theorem \ref{Thm-triangleorder}
uses two propositions, whose proofs are postponed to
the Appendix.

\begin{prop}\label{Prop-1}
Let $s,t$ be two integers. If $t\geq s\geq 1$, then
$\la_1(S(K_{s+2,t+2}))>\la_1(S(K_{s+1,t+3}))$.
\end{prop}

\begin{prop}\label{Prop-2}
Let $s,t$ be two integers. If $t\geq s\geq 1$ and
$s+t=n-5$, then
$\lambda_1(S(K_{\lfloor \frac{n-1}{2}\rfloor,\lceil \frac{n-1}{2}\rceil}))\geq \lambda_1(S(K_{s+2,t+2}))$,
where equality holds if and only if
$(s,t)=(\lfloor \frac{n-1}{2}\rfloor,\lceil \frac{n-1}{2}\rceil)$.
\end{prop}

\noindent
{\bf Proof of Theorem \ref{Thm-triangleorder}.}
Suppose that $G$ is a non-bipartite triangle-free
graph of order $n$ with the maximum spectral radius.
We shall show
$G=S(K_{\lfloor \frac{n-1}{2}\rfloor,\lceil \frac{n-1}{2}\rceil})$.
First we claim that $G$ is connected; since
otherwise we can add a new edge between a component
with the maximum spectral radius and any other
component to get a new graph with larger spectral
radius. We also observe that adding any new
edge gives us at least one triangle.

Let $x=(x_1,\ldots,x_n)^t$ be the Perron vector of
$G$ and $u$ be a vertex of $G$ with
$x_u=\max\{x_i|i=1,\ldots,n\}$. Let
$C=u_1u_2\cdots u_ku_1$ be a shortest odd cycle of
$G$ with $k\geq 5$. We have the following claims.

\begin{claim}\label{Claim-distance}
For any two vertices $x,y\in V(G)$, the distance
between $x$ and $y$ in $G$, denoted by $d_G(x,y)$,
satisfies that $d_G(x,y)\leq 2$.
\end{claim}
\begin{proof}
For any two nonadjacent vertices
$x,y\in V(G)$, let $P=v_0v_1v_2\ldots v_l$ be a
shortest $(x,y)$-path in $G$, where $v_0=x$ and
$v_l=y$. Obviously, $l\geq 2$. Since $G+xy$ is not
bipartite and $\la_1(G+xy)>\la_1(G)$, by the
choice of $G$, there is a triangle passing
through the edge $xy$ in $G+xy$. That is,
there is an $(x,y)$-path of length 2 in $G$,
and so $d_G(x,y)=2$. This proves Claim \ref{Claim-distance}.
\end{proof}

\begin{claim}\label{Claim-k=5}
$k=5$.
\end{claim}
\begin{proof}
Suppose to the contrary that $k\geq 7$. Since
$C$ is chordless, $u_1u_4\notin E(G)$. By
Claim \ref{Claim-distance}, $d_G(u_1,u_4)=2$. This
means that there exists a vertex outside $C$,
say $v$, such that $u_1vu_4$ is a path of
length 2. Then $u_1vu_4u_3u_2u_1$ is a
cycle of length 5, a contradiction.
This proves the claim.
\end{proof}
If $V(C)=V(G)$, then $G$ is an induced 5-cycle,
and $G=S(K_{2,2})$. Now assume $V(G)\backslash V(C)\neq \emptyset$
and so $n\geq 6$.

\begin{claim}\label{Claim-neighbor-w}
For each vertex $w\in V(G)\backslash (N(u)\cup V(C))$, $N(w)=N(u).$
\end{claim}
\begin{proof}
If $V(G)\backslash (N(u)\cup V(C))=\emptyset$, then
there is nothing to prove. Thus
$V(G)\backslash (N(u)\cup V(C))\neq\emptyset$.
Suppose Claim \ref{Claim-neighbor-w} is false.
Let $w\in V(G)\backslash (N(u)\cup V(C))$
such that $N(w)\neq N(u)$.
Let $G'=G-\{wv|v\in N_G(w)\}+\{wv|v\in N_G(u)\}$.
Obviously, $G'$ contains no triangles and $C$ is
also in $G'$, and so $G'$ is not bipartite.

Observe that
\begin{align}\label{Ineq-2}
\la_1(G')-\la_1(G)\geq x^t(A(G')-A(G))x\geq 2x_w\left(\sum_{v\in N(u)}x_v-\sum_{v\in N(w)}x_v\right)\geq0.
\end{align}
If $N(w)\subsetneqq N(u)$, then $\la_1(G')>\la_1(G)$,
a contradiction. Therefore,
$N(w)\backslash N(u)\neq \emptyset$ and $N(u)\backslash N(w)\neq \emptyset$.
By the choice of $G$, $\la_1(G)\geq \la_1(G')$. Thus
all inequalities of (\ref{Ineq-2}) become equalities,
and so $x$ is also the Perron vector of $G'.$ On the
other hand, choose $z\in N(u)\backslash N(w)$, we have
$\la_1(G)x_z=\sum_{v\in N_G(z)}x_v<\sum_{v\in N_G(z)\cup \{w\}}x_v=\la_1(G')x_z$,
and hence $\la_1(G)<\la_1(G')$, a contradiction.
\end{proof}

\begin{claim}\label{Claim-3.4}
For any $s,t\in N(u)\backslash V(C)$, we have
$N(s)\cap V(C)=N(t)\cap V(C)$.
\end{claim}
\begin{proof}
Let $w\in N(u)\backslash V(C)$ such that
$x_w=\max\{x_v|v\in N(u)\backslash V(C)\}$.
Note that $N(u)$ is an independent set.
Let $Y=\{v\in V(G-C)|N(v)=N(u)\}$.
By Claim \ref{Claim-neighbor-w}, there holds
$\la_1(G)x_z=\sum_{v\in Y}x_v+\sum_{v\in N(z)\cap V(C)}x_v$
for any $z\in N(u)\backslash V(C)$. This implies
that $\sum_{v\in N(w)\cap V(C)}x_v\geq \sum_{v\in N(z)\cap V(C)}x_v$
for any $z\in N(u)\backslash V(C)$.

If there exists a vertex $z'\in N(u)\backslash V(C)$ such
that $N(z')\cap V(C)\neq N(w)\cap V(C)$, let
$G'=G-\{z'v|v\in N(z')\cap V(C)\}+\{z'v|v\in N(w)\cap V(C)\}$.
Then
$\la_1(G')-\la_1(G)\geq x^t(A(G')-A(G))x\geq 2x_j\left(\sum_{v\in N(w)\cap V(C)}x_v-\sum_{v\in N(z')\cap V(C)}x_v\right)\geq0.$
One can find $G'$ is connected, non-bipartite and
triangle-free. Similar to the proof of Claim \ref{Claim-neighbor-w},
we have $\la_1(G')>\la_1(G)$, a contradiction.
Thus, for any $z\in N(u)\backslash V(C)$,
$N(z)\cap V(C)=N(w)\cap V(C)$. This proves
Claim \ref{Claim-3.4}.
\end{proof}

Let $X=N(u)\backslash V(C)$ and $Y=\{v\in V(G-C)|N(v)=N(u)\}$.
From Claim \ref{Claim-neighbor-w}, we have $V(G)=X\cup Y\cup V(C)$.
In what follows, we only need to consider three cases:
(A) $u\notin V(C)$ and $X\neq \emptyset$; (B) $u\notin V(C)$ and $X=\emptyset$;
and (C) $u\in V(C)$.

Let us first consider Case (A). For this case, $X\neq \emptyset$
and $Y\neq \emptyset$.
By Claims \ref{Claim-neighbor-w} and \ref{Claim-3.4},
we have $X\cup Y=V(G)\backslash V(C)$,
$G-C=B(X,Y)\cong K_{s,t}$, where $|X|=s$,
$|Y|=t$, $t\geq s\geq 1$. Moreover, any
two vertices in $X$ have the same neighbor
(neighbors) of $C$, and any two vertices
in $Y$ also have the same neighbor
(neighbors) of $C$.
\begin{claim}\label{Claim-neighbors}
By symmetry, we have $N_C(X)=\{u_i,u_{i+2}\}$ and
$N_C(Y)=\{u_{j},u_{j+2}\}$ such that
$N_C(X)\cap N_C(Y)=\emptyset$.
\end{claim}
\begin{proof}
We first observe that for any vertex
$v\in V(G\backslash C)$, $|N(v)\cap V(C)|\leq 2$,
since $G$ is triangle-free and $|V(C)|=5$. Next,
we shall show that $d_{C}(X)=d_{C}(Y)=2$. Since
$G$ is connected, we have $d_C(X)\geq 1$ or
$d_C(Y)\geq 1$. Assume that $d_C(X)=1$ and set
$N_C(X)=\{u_i\}$. If either $u_{i+2}$ or $u_{i-2}$
does not belong to $N_C(Y)$, then we connect such
a vertex to all vertices in $X$, and create a new
graph $G'$. Note that $G'$ is non-bipartite and
triangle-free but with larger spectral radius, a
contradiction. Hence $u_{i+2},u_{i-2}\in N_C(Y)$.
Since $C$ is a 5-cycle, $u_{i-2}u_{i+2}\in E(G)$
and there is a triangle in $G$, a contradiction.
Therefore, $d_{C}(X)=2$, and by symmetry,
$N_C(X)=\{u_i,u_{i+2}\}$. The other assertion
can be proved similarly. It follows from the
fact $G$ contains no triangles that
$N_C(X)\cap N_C(Y)=\emptyset$.
\end{proof}
By Claim \ref{Claim-neighbors} and the symmetry,
without loss of generality, we can assume
$N_C(X)=\{u_1,u_{3}\}$ and $N_C(Y)=\{u_2,u_4\}$.
By Claim \ref{Claim-k=5}, $C$ is an induced 5-cycle.
By Claims \ref{Claim-neighbor-w} and \ref{Claim-3.4}, all
vertices of $X$ are adjacent to $\{u_1,u_3\}$,
and all vertices of $Y$ are adjacent to
$\{u_2,u_4\}$. Furthermore, $u_5$ is a vertex
of degree 2 in $G$.
Observe that $G=S(K_{s+2,t+2})$. By Proposition
\ref{Prop-2} and the choice of $G$,
$G=S(K_{\lfloor\frac{n-1}{2}\rfloor,\lceil\frac{n-1}{2}\rceil})$.
This finishes Case (A).

Now we consider Case (B). For this case, 
By the fact that $G$ is connected and the choice of $G$,
we have $d_C(u)=2$. Set $N_C(u)=\{u_1,u_3\}$.
By Claim \ref{Claim-neighbor-w}, we infer
$G=S(K_{n-3,2})$. By Proposition \ref{Prop-2}, $\la_1(G)=\la_1(S(K_{n-3,2}))\leq\la_1(S(K_{\lfloor\frac{n-1}{2}\rfloor,\lceil\frac{n-1}{2}\rceil}))$£¬
where equality holds if and only if $n=6$
(recall that $n\geq 6$). Thus $G=S(K_{2,3})$
where $n=6$.

Finally we consider Case (C). By very similar
analysis as above, one can see
$G=S(K_{\lfloor\frac{n-1}{2}\rfloor,\lceil\frac{n-1}{2}\rceil})$.
The proof is complete. {\hfill$\Box$}

\section{Concluding remarks}\label{Sec:4}
In this paper, we consider the Bollob\'as-Nikiforov
Conjecture on the largest eigenvalue, the second
largest eigenvalue and size of a graph, and
settle the conjecture for triangle-free graphs,
improving the spectral version of Mantel's
theorem. We also prove two spectral analogs
of Erd\H{o}s' theorem. Many intriguing problems
with respect to this topic remain open.

\begin{itemize}
\item One can prove the following extension of Erd\H{o}s'
theorem: Let $G$ be a graph with order $n$ and the length
of odd girth at least $2k+3$. If $G$ is non-bipartite,
then $e(G)\leq \left(\frac{n-(2k-1)}{2}\right)^2+2k-1.$
The following question naturally arises: which class of
graphs can attain the maximum spectral radius among the
class of graphs above? From Theorem \ref{Thm-triangleorder},
we can see
$S(K_{\lfloor\frac{n-1}{2}\rfloor,\lceil\frac{n-1}{2}\rceil})$
is the answer for this problem when $k=1$.

\item Nosal \cite{N70} proved that every graph $G$ of size
$m$ satisfying $\lambda_1(G)>\sqrt{m}$ contains a triangle.
Nikiforov \cite[Theorem~2]{N09-2} proved that every graph
$G$ of size $m\geq 9$ contains a $C_4$ if $\lambda_1(G)>\sqrt{m}$.
Let $k$ and $m$ be two integers such that $k|m$ and $k$ is odd.
Let $S_{\frac{m}{k}+\frac{k+1}{2},k}$ be the graph obtained
by joining each vertex of $K_k$ to $\frac{m}{k}-\frac{k-1}{2}$
isolated vertices. Zhai, Lin and Shu conjectured a more general
one.

\begin{conj}[\cite{ZLS}]\label{Conj-ZLS}
Let $G$ be a graph of sufficiently large size $m$ without
isolated vertices and $k\geq 1$ be an integer. If
$\la_1(G)\geq \lambda_1(S_{\frac{m}{k}+\frac{k+1}{2},k})$,
then $G$ contains $C_t$ for every $t\leq 2k+2$ unless
$G=S_{\frac{m}{k}+\frac{k+1}{2},k}$.
\end{conj}
When $k=1$, it includes Nosal's theorem  \cite{N70}
and Nikiforov's theorem \cite[Theorem~2]{N09-2} as
two special cases.

\item By replacing $\la_1(G)$ by $s^+(G)$, some spectral
graph theorists refined and generalized many classical
results on spectral graph theory. For example, Hong's
famous theorem states that $\la_1(G)\leq \sqrt{2m-n+1}$ \cite{H93}
holds for all connected graphs $G$ with order $n$ and
size $m$. (It in fact holds for all graphs without isolated
vertices.) Elphick, Farber, Goldberg and Wocjan
conjectured that
\begin{conj}[\cite{EFGW}]\label{Conj-EFGW}
Let $G$ be a connected graph of order $n$. Then
$\min\{s^+(G),s^-(G)\}\geq n-1$.
\end{conj}
Motivated by Conjecture \ref{Conj-EFGW}, we can
reconsider Conjecture \ref{Conj-BN} in
a general form.
\begin{prob}
Let $H$ be a given graph and $G$ be an $H$-free
graph of size $m$. How to estimate the upper bound
of $s^+(G)$ in terms of $H$ and $m$?
\end{prob}
Even if $H=K_{r+1}$, this problem seems harder than
Conjecture \ref{Conj-BN}. For graphs with given
chromatic number, a related problem was studied
by Ando and Lin in \cite{AL}.
\end{itemize}

\section*{Acknowledgement}
All the revisions have been made since the second author became an associate
professor at College of Computer Science, Nankai University.
The authors are very grateful to one anonymous referee whose
many suggestions largely improve the quality of the paper.

\section{Appendix}
In this section, we shall prove Proposition \ref{Prop-1},
which can imply Proposition \ref{Prop-2}.

\noindent
{\bf{Proof of Proposition \ref{Prop-1}.}}
Set $G_1=S(K_{s+2,t+2})$ and $G_2=S(K_{s+1,t+3})$.
Since both $G_1$ and $G_2$ contain $C_5$ as a proper
subgraph, we have $\la_1(G_i)>2$ for $i=1,2.$
The characteristic polynomial of $G_1$ is
$P_{G_{1}}(x)=x^{s+t}(x^5-(2s+2t+st+5)x^3+(4s+4t+3st+5)x-2s-2t-2st-2),$
where $|V(G_1)|=s+t+5$. Let
$f(x,s,t)=x^5-(2s+2t+st+5)x^3+(4s+4t+3st+5)x-2s-2t-2st-2.$
Then $\la_1(G_1)$ is the largest root of $f(x,s,t)=0$.
Note that $f(x,s-1,t+1)-f(x,s,t)=(x-1)^2(x+2)(t-s+1)$.
Since $t\geq s$, we have $f(x,s-1,t+1)-f(x,s,t)>0$ when
$x>1$. Moreover, since $\la_1(G_2)$ is the largest
root of $f(x,s-1,t+1)=0$, it follows that
$\la_1(G_{1})>\la_1(G_{2})$. {\hfill$\Box$}

We also include a proof of the result mentioned in
Section \ref{Sec:4} here. The proof uses a result
due to Andr\'{a}sfai, Erd\H{o}s and S\'{o}s \cite{AES} to
control the minimum degree.
\begin{lemma} {\rm(\cite{AES})}\label{lem6}
Let $G$ be a graph with order $n$ and odd girth 
at least $2k+1$ where $k\geq 1$. If the
minimum degree $\delta(G)>\frac{2n}{2k+1}$, then
$G$ is bipartite.
\end{lemma}
\begin{theorem}\label{Thm-LNW}
Let $G$ be a graph with order $n$ and odd 
girth at least $2k+3$ where $k\geq 1$. If
$G$ is non-bipartite, then
$e(G)\leq \left(\frac{n-(2k-1)}{2}\right)^2+2k-1.$
\end{theorem}

\noindent
{\bf Proof of Theorem \ref{Thm-LNW}}
We prove Theorem \ref{Thm-LNW} by induction on
$n$. If $n=2k+3$, then the odd girth
is $2k+3$, and hence $G\cong C_{2k+3}$ and
$e(G)=2k+3$. The result holds. Now we assume
$n\geq 2k+4$ and the result holds for graphs
with order less than $n$.

First suppose $\delta(G)>\frac{n-(2k-1)}{2}-\frac{1}{4}$,
i.e., $\delta(G)\geq\frac{n-(2k-1)}{2}$. Since
$n>2k+3$, we have $(2k-1)n>(2k-1)(2k+3)$, which
implies $\frac{n-(2k-1)}{2}>\frac{2n}{2k+3}$. By
Lemma \ref{lem6}, $G$ is bipartite, a contradiction.
Thus $\delta(G)\leq \frac{n-(2k-1)}{2}-\frac{1}{4}$.
Let $v$ be a vertex with $d(v)=\delta(G)$ and let
$G'=G-\{v\}$.

If $G'$ is non-bipartite, then by the hypothesis, we have
$e(G')\leq(\frac{n-1-(2k-1)}{2})^2+2k-1=(\frac{n-(2k-1)}{2})^2+2k-1-\frac{n-(2k-1)}{2}+\frac{1}{4}.$
It follows $e(G)\leq \left(\frac{n-(2k-1)}{2}\right)^2+2k-1.$
Thus, $G'$ is bipartite. Then every odd cycle passes through
$v$ in $G$. Choose $C$ as a shortest odd cycle of $G$,
where $|C|\geq 2k+3$. Let $G'=B(X,Y)$, where $(X,Y)$ is
the bipartition of $G'$. Let $X_0=X\cap (V(C)-\{v\})$,
$Y_0=Y\cap (V(C)-\{v\})$, $X_1=X-X_0$ and $Y_1=Y-Y_0.$
Then
\begin{eqnarray*}
e(G)&=&e(G[C])+|E(C,G[X_1,Y_1])|+|E(G[X_1,Y_1])|\leq|C|+2|X_1|+2|Y_1|+|X_1||Y_1|\\
&=&|C|+(|X_1|+2)(|Y_1|+2)-4\\
&\leq&|C|+\left(\frac{|X_1|+2+|Y_1|+2}{2}\right)^2-4\\
&=&|C|+\left(\frac{n-|C|+4}{2}\right)^2-4.
\end{eqnarray*}
Let $f(x)=x+(\frac{n-x+4}{2})^2-4$. Then $f'(x)=1-(n-x+4)=x-(n+3)$.
Hence for $2k+3\leq x\leq n$, we have
$f(x)\leq f(2k+3)=\left(\frac{n-(2k-1)}{2}\right)^2+2k-1$.
The proof is complete. {\hfill$\Box$}
\begin{remark}
Theorem \ref{Thm-LNW}
maybe appeared in some reference. Since we cannot
find such one till now, we present a proof here
for completeness.
\end{remark}

\end{sloppypar}
\end{document}